\documentclass[11pt]{amsart}
\usepackage{amsthm, amsfonts, amssymb, amsmath, bbold, enumitem, xcolor, url, stmaryrd}
\usepackage{avant}

\newcommand{\Q}{\mathbb{Q}}

\newcommand{\F}{\mathbb{F}}

\newcommand{\Fq}{\F_{\!q}}
\newcommand{\Fp}{\F_{\!p}}

\newtheorem{thm}{Theorem}[section]
\newtheorem{lem}[thm]{Lemma}

\theoremstyle{definition}
\newtheorem{ex}{Example}[section]
\newtheorem{rem}{Remark}[section]

\begin{document}
\title{On quadratic character sums over quartics}

\begin{abstract}
We obtain transformation formulas for quadratic character sums with quartic and cubic polynomial arguments. 

\end{abstract}

\date{\today}
\author{Bogdan Nica}

\address{\newline Department of Mathematical Sciences \newline Indiana University Indianapolis}
\email{bnica@iu.edu}

\maketitle

\section{Introduction}
Let $\sigma$ denote the quadratic character on $\Fq$, the finite field with $q$ elements. Throughout, we assume that $\mathrm{char} \: \Fq\neq 2,3$. This note is concerned with quadratic character sums of the form
\begin{align}\label{eq0}
\sum_{x\in \Fq} \sigma(f(x)).\tag{$S_f$}
\end{align}
Here $f(x)$ is a monic polynomial with coefficients in $\Fq$, usually required to be square-free. The quadratic character sum \eqref{eq0} is intimately related to the number of $\Fq$-points on the curve $y^2=f(x)$:
\[\#\big\{(x,y)\in \Fq\times \Fq: y^2=f(x)\big\}=q+\sum_{x\in \Fq} \sigma(f(x)).\]

The evaluation of the quadratic character sum \eqref{eq0} is easily settled in the case when the monic polynomial $f(x)$ is linear or quadratic:
\begin{align}
\sum_{x\in \Fq} \sigma(f(x))=\begin{cases}
0 & \textrm{ if } \deg f=1,\\
-1 & \textrm{ if } \deg f=2.
\end{cases}
\end{align}

When $f(x)$ has degree $3$ or higher, explicit evaluations are known only for sporadic families of polynomials. Still, there are several viewpoints that afford considerable insight. How does a quadratic character sum \eqref{eq0} \emph{grow}, in terms of the degree of $f$? This is the search for estimates, a celebrated example being the Hasse--Weil bound. How does a quadratic character sum \eqref{eq0} \emph{vary}, as we change (some of) the coefficients for $f$? This is the statistical viewpoint; we owe to Birch \cite{B} an early work in this rich vein, but see also \cite{N2}. How do quadratic character sums of the form \eqref{eq0} \emph{relate} to each other? This is the pursuit of transformation formulas--namely, identities between quadratic character sums associated to two different polynomials. Skillful usage of such transformation formulas often leads to new explicit evaluations of quadratic character sums.

An illustrative example of a transformation formula for quadratic character sums is the following: for all $b,c\in \Fq$ we have
\begin{align}\label{eq: jac}
\sum_{x\in \Fq} \sigma(x)\sigma(x^2+bx+c)=\sum_{x\in \Fq} \sigma(x+b)\sigma(x^2-4c).
\end{align}
This cubic-to-cubic transformation formula was first obtained by Jacobsthal \cite{Jac}, and rediscovered in \cite[Thm.1]{PAR}. Compare also the prime case of \cite[Thm.1.2]{KPSV}.

Another example is the following transformation formula: for all $a\in \Fq$ we have
\begin{align}\label{eq: lm}
\sum_{x\in \Fq} \sigma(x^5+ax^3+x)=\big(1+\sigma(-1)\big)\sum_{x\in \Fq} \sigma\big(x^3+4x^2+(a+2)x\big).
\end{align}
This is due to Lepr\'evost and Morain \cite[Thm.1]{LM}. See \cite[Exer.5.29, Thm.5.20]{N} for generalizations. Compare also results in \cite[Secs.4,5]{KPSV}. The transformation formula \eqref{eq: lm} is a descent formula: it lowers the degree of the polynomial argument--in some sense, its complexity--from quintic, on the left-hand side, to cubic, on the right-hand side.

\section{Transformation formulas over quartics}
In this note, we are interested in transformation formulas for quadratic character sums of the form \eqref{eq0}, in which $f(x)$ is a quartic polynomial. 

A simple, yet useful example is the following descent formula for biquadratics: if $b,c\in \Fq$ satisfy $b^2\neq 4c$ then
\begin{align}\label{eq: biq}
\sum_{x\in \Fq}\sigma(x^4+bx^2+c)=-1+\sum_{x\in \Fq}\sigma(x^3+bx^2+cx).
\end{align}
This is an instance of a more general descent principle: for any polynomial $g(x)$ we have
\begin{align*}
\sum_{x\in \Fq}\sigma(g(x^2))=\sum_{x\in \Fq}(1+\sigma(x))\sigma(g(x))=\sum_{x\in \Fq}\sigma(g(x))+\sum_{x\in \Fq}\sigma(xg(x)).
\end{align*}
Taking $g(x)=x^2+bx+c$, in which case $\sum_{x\in \Fq}\sigma(g(x))=-1$, we get \eqref{eq: biq}. 

The next descent formula was obtained, in a slightly different form, by Williams \cite{W79}; see also \cite[Thm.5.10]{N}. Compared to \eqref{eq: biq}, its algebraic origin is harder to detect.

\begin{thm}\label{thm: desc2}
Let $f(x)=(x^2+b_1x+c_1)(x^2+b_2x+c_2)$ be square-free. Then
\begin{align}\label{eq: desc2intro}
\sum_{x\in \Fq} \sigma(f(x))=-1+\sum_{x\in \Fq} \sigma\big(x^3+Bx^2+\Delta_1\Delta_2x\big)
\end{align}
where $\Delta_1=b_1^2-4c_1$, $\Delta_2=b_2^2-4c_2$, and $B=4(c_1+c_2)-2b_1b_2$.
\end{thm}

There is, in fact, a general quartic-to-cubic descent formula. It reads as follows.

\begin{thm}\label{thm: desc}
Let $f(x)=x^4+a_3x^3+a_2x^2+a_1x+a_0$ be square-free. Then
\begin{align}\label{eq: gendescintro}
\sum_{x\in \Fq} \sigma(f(x))=-1+\sum_{x\in \Fq} \sigma(g(x))
\end{align}
where $g(x)=x^3+a_2x^2+(a_1a_3-4a_0)x+a_0(a_3^2-4a_2)+a_1^2$.
\end{thm}

Despite its generality, the descent formula \eqref{eq: gendescintro} does not supplant other quartic-to-cubic descent formulas such as \eqref{eq: biq} or \eqref{eq: desc2intro}. For neither one of \eqref{eq: biq} or \eqref{eq: desc2intro} is an instance of \eqref{eq: gendescintro}, as one might naively hope; and while they can be eventually derived from \eqref{eq: gendescintro}, additional ingredients are needed. It turns out that Jacobsthal's cubic-to-cubic formula \eqref{eq: jac} gets used in deriving both \eqref{eq: biq} from \eqref{eq: gendescintro}, as well as \eqref{eq: desc2intro} from \eqref{eq: gendescintro}. The latter derivation, \eqref{eq: desc2intro} from \eqref{eq: gendescintro}, is spelled out in Remark~\ref{rem}.

We have first arrived at Theorem~\ref{thm: desc} by interpreting, in the language of quadratic character sums, a well-known fact from the theory of elliptic curves--that a quartic curve $y^2=f(x)$ can be turned into a cubic curve $y^2=g(x)$ by means of a rational change of variables over the coefficient field. This fact is seemingly due to Mordell \cite[p.77]{Mor}. Here is a convenient version of Mordell's procedure, after Nagao \cite[p.153]{Nag}: the curve $y^2=x^4+a_3x^3+a_2x^2+a_1x+a_0$ turns into $Y^2=X^3+a_2X^2+(a_1a_3-4a_0)X+ a_0(a_3^2-4a_2)+a_1^2$ under the change of variables
\[x:=\frac{2(Y-a_1)-a_3X}{4(X+a_2)-a_3^2}, \qquad y:=\frac{2x^2+a_3x-X}{2}.\]
There is a visible issue at $X=-a_2+a_3^2/4$, which ends up accounting for the $-1$ term in \eqref{eq: gendescintro}.

The search for a direct proof of Theorem~\ref{thm: desc}, which avoids elliptic curves and rational transformations, led us to a rather general transformation formula. The `master formula' requires some notational set-up so we will not state it here. Let us emphasize instead its flexibility as a unified tool. Generically, the master formula yields quartic-to-quartic transformation formulas, a sample result being the following.

\begin{thm}\label{thm: qq}
Let $a,b,c\in \Fq$ with $c\neq 0$. Then
\begin{align*}
\sum_{x\in \Fq} \sigma\big(x^4+(2a-4b)x^2-4cx&+a^2\big)=\sum_{x\in \Fq} \sigma\big(x^4+(2b-4a)x^2-4cx+b^2\big).
\end{align*}
\end{thm}
But we can also use the master formula in order to derive the quartic-to-cubic descent formulas of Theorem~\ref{thm: desc} and Theorem~\ref{thm: desc2}, and even cubic-to-cubic transformation formulas.

We obtain the master formula in Section~\ref{M}, and then we derive the transformation formulas in Section~\ref{A}. In Section~\ref{E} we discuss some concrete examples. 

\section{The master formula}\label{M}
The starting point is the idea that a transformation formula is an identity arising from double-counting. Consider a quartic form in two variables $u$ and $x$, with coefficients in $\Fq$, given by
\begin{align}
F(u,x)=  \begin{pmatrix} u^2 \\ u \\ 1\end{pmatrix}^T
\begin{pmatrix}
a_1 & a_2 &a_3\\
b_1 & b_2 &b_3\\
c_1 & c_2 &c_3
\end{pmatrix} \begin{pmatrix} x^2 \\ x \\ 1\end{pmatrix}.
\end{align}
The quartic $F(u,x)$ can be expanded in two ways: as a quadratic in $u$,
\begin{align}\label{eq: Finu}
F(u,x)=\alpha(x)u^2+\beta(x)u+\gamma(x)
\end{align}
respectively as a quadratic in $x$,
\begin{align}\label{eq: Finx}
F(u,x)=\delta_1(u)x^2+\delta_2(u)x+\delta_3(u).
\end{align}
Here
\begin{align}\label{pols}
\begin{cases}
\alpha(x)=a_1x^2+a_2x+a_3\\
\beta(x)=b_1x^2+b_2x+b_3\\
\gamma(x)=c_1x^2+c_2x+c_3\\
\end{cases},\qquad 
\begin{cases}
\delta_1(u)=a_1u^2+b_1u+c_1\\
\delta_2(u)=a_2u^2+b_2u+c_2\\
\delta_3(u)=a_3u^2+b_3u+c_3\\
\end{cases}.
\end{align}
Mnemonically, $\alpha$, $\beta$, $\gamma$ are the \emph{row polynomials}, while $\delta_1$, $\delta_2$, $\delta_3$ are the \emph{down polynomials}. The $3\times 3$ matrix underlying the quartic form $F(u,x)$ is referred to as the \emph{coefficient matrix}.

We wish to count, in two ways, the number of solutions $(u,x)\in \Fq\times \Fq$ to the equation $F(u,x)=0$. To do so, we need the following lemma.

\begin{lem} Let $f(x)$, $g(x)$, $h(x)$ be polynomials with coefficients in $\Fq$. Then the number of solutions $(x,y)\in \Fq\times \Fq$ to the equation
\begin{align}\label{eq: ly}
f(x)y^2+g(x)y+h(x)=0
\end{align}
is given by the formula
\begin{align}\label{eq: ny}
q\cdot (1+n_{f,g,h})-n_f+\sum_{x\in \Fq} \sigma\big(g(x)^2-4f(x)h(x)\big)
\end{align}
where $n_f=\#\{x\in \Fq: f(x)=0\}$ is the number of zeros of $f(x)$, and $n_{f,g,h}=\#\{x\in \Fq: f(x)=g(x)=h(x)=0\}$ is the number of common zeros of $f(x)$, $g(x)$, $h(x)$.
\end{lem}

\begin{proof}
For each $x\in \Fq$ which satisfies $f(x)\neq 0$, the quadratic equation \eqref{eq: ly} has $1+\sigma\big(g(x)^2-4f(x)h(x)\big)$ solutions $y\in \Fq$. The contribution to the solution count in this generic case is
\begin{align*}
\sum_{x\in \Fq: f(x)\neq 0}\Big(1+ \sigma\big(g(x)^2-4f(x)h(x)\big)\Big),
\end{align*}
which can be rewritten as 
\begin{align}\label{eq: ny1}
q-n_f- \sum_{x\in \Fq: f(x)=0}\sigma\big(g(x)^2\big)+\sum_{x\in \Fq} \sigma\big(g(x)^2-4f(x)h(x)\big).
\end{align}

Consider now the contribution coming from those $x\in \Fq$ which are zeros of $f(x)$. Then \eqref{eq: ly}  turns into the linear equation $g(x)y+h(x)=0$. This has one solution $y\in \Fq$ whenever $g(x)\neq 0$, respectively $q$ solutions whenever $g(x)=h(x)=0$. We can capture the solution count in this singular case by the formula
\begin{align}\label{eq: ny2}
q\cdot n_{f,g,h}+\sum_{x\in \Fq: f(x)=0}\sigma\big(g(x)^2\big).
\end{align}
Adding up the counts \eqref{eq: ny1} and \eqref{eq: ny2}, we obtain \eqref{eq: ny}.
\end{proof}

We apply the lemma to the two equations, $\alpha(x)u^2+\beta(x)u+\gamma(x)=0$ and $\delta_1(u)x^2+\delta_2(u)x+\delta_3(u)=0$. The solution count is the same since, we recall, they both represent the equation $F(u,x)=0$. We deduce the following.

\begin{thm}[Master formula]\label{thm: TF}
Let $\alpha(x)$, $\beta(x)$, $\gamma(x)$ and $\delta_1(u)$, $\delta_2(u)$, $\delta_3(u)$ be the polynomials given by \eqref{pols}. Then
\begin{align*}
q\cdot n_{\alpha,\beta,\gamma}&-n_{\alpha}+\sum_{x\in \Fq} \sigma\big(\beta(x)^2-4\alpha(x)\gamma(x)\big)=\\
& q\cdot n_{\delta_1,\delta_2,\delta_3}-n_{\delta_1}+\sum_{u\in \Fq} \sigma\big(\delta_2(u)^2-4\delta_1(u)\delta_3(u)\big).
\end{align*}
\end{thm}

Theorem~\ref{thm: TF} reads as a general transformation formula relating two quadratic character sums whose polynomial argument is at most quartic. The heart of the matter in applying Theorem~\ref{thm: TF} is a suitable choice of a coefficient matrix.

In each one of the applications of Theorem~\ref{thm: TF} that follow, we will end up having $n_{\alpha,\beta,\gamma}=0$ and $n_{\delta_1,\delta_2,\delta_3}=0$. In some cases, this will be seen directly. In other cases, the following observation will prove useful.

\begin{lem}\label{lem: sf}
Assume that $\beta(x)^2-4\alpha(x)\gamma(x)$ is square-free. Then $n_{\alpha,\beta,\gamma}=0$ and $n_{\delta_1,\delta_2,\delta_3}=0$.
\end{lem}

\begin{proof}
If $n_{\alpha,\beta,\gamma}>0$ then $\alpha(x)$, $\beta(x)$, and $\gamma(x)$ have a common linear factor $x-x_0$. Whence $\beta(x)^2-4\alpha(x)\gamma(x)$ is divisible by $(x-x_0)^2$, in contradiction with the square-free hypothesis.

If $n_{\delta_1,\delta_2,\delta_3}>0$ then $\delta_1(u)$, $\delta_2(u)$, and $\delta_3(u)$ have a common zero $u_0$. From \eqref{eq: Finx} we see that $F(u_0,x)=0$, as a polynomial in $x$. In turn this implies, thanks to \eqref{eq: Finu}, that $-\gamma(x)=u_0\beta(x)+u_0^2\alpha(x)$ as polynomials. We then get $\beta(x)^2-4\alpha(x)\gamma(x)=(\beta(x)+2u_0\alpha(x))^2$, in contradiction with the square-free hypothesis. 
\end{proof}

Henceforth, there is no need to distinguish the two variables in Theorem~\ref{thm: TF}; the variable $u$ will be relabeled as $x$. 

\section{Applications}\label{A}
\subsection{A symmetric formula}
For the coefficient matrix 
\begin{align*}
\begin{pmatrix}
0 & 1 & 0\\
1 & 0 & a\\
0 & b &c
\end{pmatrix}
\end{align*}
the associated polynomials are
\begin{align*}
\begin{cases}
\alpha(x)=x\\
\beta(x)=x^2+a\\
\gamma(x)=bx+c\\
\end{cases},\qquad 
\begin{cases}
\delta_1(x)=x\\
\delta_2(x)=x^2+b\\
\delta_3(x)=ax+c\\
\end{cases}.
\end{align*}
We have $n_\alpha=1$ and $n_{\delta_1}=1$. 

Assuming that $(a,c)\neq (0,0)$ and $(b,c)\neq (0,0)$, we also have $n_{\alpha,\beta,\gamma}=0$ and $n_{\delta_1,\delta_2,\delta_3}=0$. Theorem~\ref{thm: TF} yields the pleasantly symmetric, quartic-to-quartic transformation formula
\begin{align}\label{eq: qq0}
\sum_{x\in \Fq} \sigma\big((x^2+a)^2-4x(bx+c)\big)=\sum_{x\in \Fq} \sigma\big((x^2+b)^2-4x(ax+c)\big).
\end{align}

The case $c\neq 0$ is Theorem~\ref{thm: qq}.

The case $c=0$ is not without interest. Now $a,b\neq 0$, and \eqref{eq: qq0} involves biquadratic arguments. By applying the descent formula \eqref{eq: biq} to both sides, we obtain the cubic-to-cubic transformation formula
\begin{align}
\sum_{x\in \Fq} \sigma(x)\sigma\big((x+a)^2-4bx\big)=\sum_{x\in \Fq} \sigma(x)\sigma\big((x+b)^2-4ax\big).
\end{align}

\subsection{A general descent formula}
Consider the coefficient matrix 
\begin{align*}
\begin{pmatrix}
0 & 0 & -1/4\\
1 & a_3/2 & 0\\
a_2-a_3^2/4 & a_1 &a_0
\end{pmatrix}.
\end{align*}
The associated polynomials are
\begin{align*}
\begin{cases}
\alpha(x)=-1/4\\
\beta(x)=x^2+(a_3/2)x\\
\gamma(x)=(a_2-a_3^2/4)x^2+a_1x+a_0\\
\end{cases},\qquad 
\begin{cases}
\delta_1(x)=x+(a_2-a_3^2/4)\\
\delta_2(x)=(a_3/2)x+a_1\\
\delta_3(x)=-x^2/4+a_0\\
\end{cases}.
\end{align*}
Visibly, $n_\alpha=0$ and $n_{\delta_1}=1$. We compute
\[
\begin{cases}
\beta(x)^2-4\alpha(x)\gamma(x)=x^4+a_3x^3+a_2x^2+a_1x+a_0,\\
\delta_2(x)^2-4\delta_1(x)\delta_3(x)=x^3+a_2x^2+(a_1a_3-4a_0)x+a_0(a_3^2-4a_2)+a_1^2.
\end{cases}
\]

Theorem~\ref{thm: TF}, with a bit of help from Lemma~\ref{lem: sf}, yields Theorem~\ref{thm: desc}.

\begin{rem}
In the case of a depressed quartic, Theorem~\ref{thm: desc} simplifies to the following descent formula: if $f(x)=x^4+ax^2+bx+c$ is square-free then
\begin{align}\label{eq: ddesc}
\sum_{x\in \Fq} \sigma(f(x))=-1+\sum_{x\in \Fq} \sigma(x^3+ax^2-4cx+b^2-4ac).
\end{align}

Conversely, Theorem~\ref{thm: desc} can be recovered by employing \eqref{eq: ddesc} as follows: given a quartic $f(x)=x^4+a_3x^3+a_2x^2+a_1x+a_0$, depress it by the variable shift $x:=x-a_3/4$; next, apply \eqref{eq: ddesc} to the resulting depressed quartic; finally, in the resulting cubic make the variable shift $x:=x+a_3^2/8$.
\end{rem}

\subsection{Descent for products of quadratics}
Consider the coefficient matrix 
\begin{align*}
\begin{pmatrix}
1 & b_1 & c_1\\
0 & 0 & 0\\
-1 & -b_2 &-c_2
\end{pmatrix}.
\end{align*}
The associated polynomials are
\begin{align*}
\begin{cases}
\alpha(x)=x^2+b_1x+c_1\\
\beta(x)=0\\
\gamma(x) =-(x^2+b_2x+c_2)\\
\end{cases},\qquad 
\begin{cases}
\delta_1(x)=x^2-1\\
\delta_2(x)=b_1x^2-b_2\\
\delta_3(x)=c_1x^2-c_2\\
\end{cases}.
\end{align*}
We then have 
\[\begin{cases}
\beta(x)^2-4\alpha(x)\beta(x) =4(x^2+b_1x+c_1)(x^2+b_2x+c_2)\\
\delta_2(x)^2-4\delta_1(x)\delta_3(x)=\Delta_1x^4+Bx^2+\Delta_2
\end{cases}\]
where $\Delta_1=b_1^2-4c_1$, $\Delta_2=b_2^2-4c_2$, $B=4(c_1+c_2)-2b_1b_2$. Note also that $n_\alpha=1+\sigma(\Delta_1)$ and $n_{\delta_1}=2$. 

We require that the quartic polynomial $f(x)=(x^2+b_1x+c_1)(x^2+b_2x+c_2)$ be square-free. Then $n_{\alpha,\beta,\gamma}=0$ and $n_{\delta_1,\delta_2,\delta_3}=0$ thanks to Lemma~\ref{lem: sf}. Thus far, Theorem~\ref{thm: TF} yields the identity
\begin{align}\label{eq: a31}
\sum_{x\in \Fq} \sigma(f(x))=-1+\sigma(\Delta_1)+\sum_{x\in \Fq} \sigma\big(\Delta_1x^4+Bx^2+\Delta_2\big).
\end{align}

Our assumption that $f(x)=(x^2+b_1x+c_1)(x^2+b_2x+c_2)$ be square-free amounts to $\Delta_1\neq 0$, $\Delta_2\neq 0$, and $B^2\neq 4\Delta_1\Delta_2$. Indeed, square-freeness of $f(x)$ means that its discriminant $\Delta(f)$ is non-zero. Now $\Delta(f)=\Delta_1\cdot \Delta_2\cdot \Pi^2$ where $\Pi=(z_1-z_2)(z_1-w_2)(w_1-z_2)(w_1-w_2)$. In the latter formula, $z_1,w_1$ are the roots of $x^2+b_1x+c_1$, and $z_2,w_2$ are the roots of $x^2+b_2x+c_2$, all in some extension of $\Fq$. We can evaluate $\Pi$ in terms of the coefficients $b_1$, $b_2$, $c_1$, $c_2$ as follows: 
\begin{align*}
\Pi&=(z_1^2+b_2z_1+c_2)(w_1^2+b_2w_1+c_2)\\
&=(c_1-c_2)^2-b_1b_2(c_1+c_2)+b_1^2c_2+b_2^2c_1.
\end{align*}
A further calculation reveals that $16\Pi=B^2-4\Delta_1\Delta_2$. To conclude, $\Delta(f)\neq 0$ amounts to $\Delta_1\neq 0$, $\Delta_2\neq 0$, and $B^2-4\Delta_1\Delta_2\neq 0$. 

As $\Delta_1\neq 0$ and $B^2\neq 4\Delta_1\Delta_2$, we infer from \eqref{eq: biq} that 
\begin{align}\label{eq: a32}
\sum_{x\in \Fq} \sigma\big(\Delta_1x^4+Bx^2+\Delta_2\big)=-\sigma(\Delta_1)+\sum_{x\in \Fq} \sigma\big(\Delta_1x^3+Bx^2+\Delta_2x\big).
\end{align}
Furthermore, the change of variable $x:=x/\Delta_1$ in the latter sum gives
\begin{align}\label{eq: a33}
\sum_{x\in \Fq} \sigma\big(\Delta_1x^3+Bx^2+\Delta_2x\big)=\sum_{x\in \Fq} \sigma\big(x^3+Bx^2+\Delta_1\Delta_2x\big).
\end{align}

Taken together, the identities \eqref{eq: a31}, \eqref{eq: a32}, \eqref{eq: a33} yield
\begin{align*}
\sum_{x\in \Fq} \sigma(f(x))=-1+\sum_{x\in \Fq} \sigma\big(x^3+Bx^2+\Delta_1\Delta_2x\big).
\end{align*}
We have thereby obtained the descent formula \eqref{eq: desc2intro} of Theorem~\ref{thm: desc2}.

\begin{rem}\label{rem}
As already mentioned, it is possible to obtain the descent formula \eqref{eq: desc2intro} from the general quartic-to-cubic formula \eqref{eq: gendescintro}. Here are the steps: (i) reduce the proof of \eqref{eq: desc2intro} to its depressed case by using a variable shift $x:=x+e$, (ii) use \eqref{eq: ddesc}, the depressed instance of \eqref{eq: gendescintro}, and (iii) use Jacobsthal's cubic-to-cubic formula \eqref{eq: jac} to conclude.

A change of variable $x:=x+e$ turns $f(x)=(x^2+b_1x+c_1)(x^2+b_2x+c_2)$ into a polynomial of the same form, $f^*(x)=(x^2+b^*_1x+c^*_1)(x^2+b^*_2x+c^*_2)$. Crucially, the parameters involved in the right-hand side of \eqref{eq: desc2intro} remain unchanged: $\Delta^*_1=\Delta_1$, $\Delta^*_2=\Delta_2$, and $B^*=B$. The first two are evident; the latter requires a computation, left to the reader. In view of this shifting invariance, we may choose a convenient shift, namely, $e=-(b_1+b_2)/4$, which depresses $f(x)$. It therefore suffices to obtain \eqref{eq: desc2intro} when $b_1+b_2=0$, in other words $b_2=-b_1=:b$.

In the case of a square-free polynomial of the form $f^*(x)=(x^2-bx+c_1)(x^2+bx+c_2)$, the quartic-to-cubic formula \eqref{eq: ddesc} says that
\begin{align*}
\sum_{x\in \Fq} \sigma(f^*(x))=-1+\sum_{x\in \Fq} \sigma(g(x))
\end{align*}
where $g(x)=x^3+(c_1+c_2-b^2)x^2-4c_1c_2x+b^2(c_1+c_2)^2-4c_1c_2(c_1+c_2)$.

We factor $g(x)=(x+c_1+c_2)\big(x^2-b^2x+b^2(c_1+c_2)-4c_1c_2\big)$. The change of variable $x:=x+b^2/2$ leads to
\begin{align*}
\sum_{x\in \Fq} \sigma(g(x))=\sum_{x\in \Fq} \sigma\bigg(\Big(x+\frac{B^*}{4}\Big)\Big(x^2-\frac{\Delta_1^*\Delta_2^*}{4}\Big)\bigg)
\end{align*}
where $B^*=4(c_1+c_2)+2b^2$, $\Delta_1^*=b^2-4c_1$, $\Delta_2^*=b^2-4c_2$ are the expected parameters. Next, we use the rescaling $x:=x/4$, followed by \eqref{eq: jac}, to get
\begin{align*}
\sum_{x\in \Fq} \sigma(g(x))&=\sum_{x\in \Fq} \sigma\big((x+B^*)(x^2-4\Delta_1^*\Delta_2^*)\big)\\
&=\sum_{x\in \Fq} \sigma\big(x(x^2+B^*x+\Delta_1^*\Delta_2^*)\big).
\end{align*}
Combining these steps, we conclude that
\begin{align*}
\sum_{x\in \Fq} \sigma(f^*(x))=-1+\sum_{x\in \Fq} \sigma\big(x(x^2+B^*x+\Delta_1^*\Delta_2^*)\big).
\end{align*}
\end{rem}

\section{Explicit examples}\label{E}
The parametric form of the transformation formulas discussed above confers them generality and applicability. In this final section we consider some explicit examples. We work over a prime field $\Fp$, where $p>3$. The quadratic character $\sigma$ is now written classically, as the Legendre symbol $(\cdot/p)$. 

For concrete applications, the most effective transformation formula is usually \eqref{eq: gendescintro} or its depressed case, \eqref{eq: ddesc}. Explicit evaluations are known only for sporadic families of cubics or quartics but, comparatively speaking, more is known for cubic arguments than for quartic arguments. So, if we are aiming for an explicit evaluation, then a good strategy is to promptly use a descent formula and turn a given quartic polynomial argument into a cubic. With a bit of luck, we just might land on a known evaluation--possibly after some simple shifting and rescaling.


\begin{ex}\label{ex1}
We evaluate the quadratic character sum
\begin{align}
S:=\sum_{x\in \Fp}\left(\frac{x^4+14x^3+24x^2+14x+1}{p}\right).
\end{align}
The quartic argument is square-free, as its discriminant is $-2^6\cdot 3^9\neq 0$. Theorem~\ref{thm: desc} gives
\begin{align*}
S=-1+\sum_{x\in \Fp}\left(\frac{x^3+24x^2+192x+296}{p}\right).
\end{align*}
The cubic argument can be written as $(x+8)^3-6^3$. A shift $x:=x-8$, and then a scaling $x:=-6x$, give
\begin{align*}
S=-1+\sum_{x\in \Fp}\left(\frac{x^3-6^3}{p}\right)=-1+\left(\frac{-6}{p}\right)\sum_{x\in \Fp}\left(\frac{x^3+1}{p}\right).
\end{align*}
The latter sum is a well-known cubic Jacobsthal sum. It evaluates as follows:
\begin{align*}
\sum_{x\in \Fp} \left(\frac{x^3+1}{p}\right)=
\begin{cases}
2A_3 & \textrm{ if } p\equiv 1 \textrm{ mod } 3,\\
0 & \textrm{ if } p\equiv 2 \textrm{ mod } 3,
\end{cases}
\end{align*}
where, in the case $p\equiv 1$ mod $3$, we write $p=A_3^2+3B_3^2$ with the convention that $A_3\equiv -1$ mod $3$. See, for instance, \cite[Chap.3]{N}.

As $(-3/p)=1$ when $p\equiv 1$ mod $3$, we conclude that
\begin{align}
S=\begin{cases}
-1+ (2/p)\cdot 2A_3 & \textrm{ if } p\equiv 1 \textrm{ mod } 3,\\
-1 & \textrm{ if } p\equiv 2 \textrm{ mod } 3.
\end{cases}
\end{align}
\end{ex}

\begin{ex}
We evaluate the quadratic character sum
\begin{align}
S:=\sum_{x\in \Fp} \left(\frac{x^4+8x^3+24x^2-44x+16}{p}\right).
\end{align}
We first depress the quartic argument by the change of variable $x:=x-2$. We get the simpler quartic $x^4-76x+152$, whose discriminant is $-2^8\cdot 19^3$. For $p=19$, we easily get $S=\sum_{x\in \Fp} (x^4/p)=p-1$.

Assume $p\neq 19$. We may then use \eqref{eq: ddesc} to get
\[S=-1+\sum_{x\in \Fp} \left(\frac{x^3-608 x+5776}{p}\right).\]
The latter quadratic character sum can be evaluated by the following formula due to Rishi, Parnami, and Rajwade \cite{RPR}: for $\lambda\in \Fp^*$, we have
\begin{align*}
\sum_{x\in \Fp} \left(\frac{x^3-2^3\cdot 19\cdot  \lambda^2 x+2\cdot 19^2\cdot \lambda^3}{p}\right)=
\begin{cases}
(2\lambda /p) \cdot (A/19)\cdot A & \textrm{ if } (p/19)=1,\\ 
0  & \textrm{ if } (p/19)=-1,
\end{cases}
\end{align*}
where, in the case $(p/19)=1$, we write $4p=A^2+19B^2$.

Taking $\lambda=2$, we conclude that
\begin{align}
S=\begin{cases}
-1+(A/19)\cdot A & \textrm{ if } (p/19)=1,\\ 
-1  & \textrm{ if } (p/19)=-1.
\end{cases}
\end{align}
\end{ex}

\begin{ex} \label{ex3}
In a very recent preprint \cite{Z}, Wenpeng Zhang pointed out that two different computations of the fourth moment associated to certain generalized Kloosterman sums imply the identity
\[
\left(\frac{-1}{p}\right)\sum_{x\in \Fp^*} \left(\frac{x^3+x^2+x}{p}\right)-\sum_{x\in \Fp^*} \left(\frac{(x^2+1)(x^2+4x+1)}{p}\right)=2.
\]
Zhang asked for a direct proof of the identity. We give one below. 

After including the terms corresponding to $x=0$, the above identity can be rewritten as 
\begin{align}\label{eq: wz}
\sum_{x\in \Fp} \left(\frac{(x^2+1)(x^2+4x+1)}{p}\right)=-1+\left(\frac{-1}{p}\right) \sum_{x\in \Fp} \left(\frac{x^3+x^2+x}{p}\right).
\end{align}
We have two quartic-to-cubic descent formulas at our disposal, Theorem~\ref{thm: desc2} and Theorem~\ref{thm: desc}. Using Theorem~\ref{thm: desc2} is a natural choice, and \eqref{eq: wz} can indeed be obtained in this way. But Theorem~\ref{thm: desc} turns out to offer the quicker road. 

Note that $(x^2+1)(x^2+4x+1)=(x+1)^4-4x^2$. The change of variable $x:=x-1$ gives
\begin{align*}
\sum_{x\in \Fp}\left(\frac{(x^2+1)(x^2+4x+1)}{p}\right)= \sum_{x\in \Fp}\left(\frac{x^4-4(x-1)^2}{p}\right).
\end{align*}
The quartic argument is square-free, as it has discriminant $-2^{12}\cdot 3\neq 0$. We apply \eqref{eq: ddesc}--that is, the depressed case of Theorem~\ref{thm: desc}--to the latter sum:
\begin{align*}
\sum_{x\in \Fp} \left(\frac{x^4-4(x-1)^2}{p}\right)=-1+\sum_{x\in \Fp} \left(\frac{x^3-4x^2+16x}{p}\right).
\end{align*}
Finally, the rescaling $x:=-4x$ gives
\begin{align*}
\sum_{x\in \Fp} \left(\frac{x^3-4x^2+16x}{p}\right)= \left(\frac{-1}{p}\right)\sum_{x\in \Fp} \left(\frac{x^3+x^2+x}{p}\right).
\end{align*}
We deduce \eqref{eq: wz} by combining the above identities.
\end{ex}

The quartic polynomials of Example~\ref{ex1} and Example~\ref{ex3} are palindromic--that is to say, of the form $x^4+ax^3+bx^2+ax+1$. A general descent formula for palindromic quartics is discussed in \cite[Ex.5.18]{N}. We take this opportunity to point out an error in \cite[Ex.5.18]{N}: in the case $b=6$, the answer should be $-1+\varphi_2(16-a^2)$ instead of $-1+\varphi_2(1)$. The root error occurs earlier, in equation (138) of \cite[Thm.5.16]{N}: in the case $s=0$, the answer should be $-1+\varphi_2(cr)$ instead of $-1+\varphi_2(r)$.

For further results in the calculus of quadratic character sums, we refer the reader to \cite[Chap.5]{N}.



\begin{thebibliography}{88}
\bibitem{B} B.J. Birch: \emph{How the number of points of an elliptic curve over a fixed prime field varies}, J. London Math. Soc. 43 (1968), 57--60


\bibitem{Jac} E. Jacobsthal: \emph{Anwendungen einer Formel aus der Theorie der quadratischen Reste}, Dissertation, Berlin, 1906


\bibitem{KPSV} D. Krachun, F. Petrov, Z.-W. Sun, M. Vsemirnov: \emph{On some determinants involving Jacobi symbols}, Finite Fields Appl. 64 (2020), Paper no. 101672

\bibitem{LM} F. Lepr\'evost, F. Morain: \emph{Rev\^{e}tements de courbes elliptiques \`{a} multiplication complexe par des courbes hyperelliptiques et sommes de caract\`{e}res}, J. Number Theory 64 (1997), no. 2, 165--182

\bibitem{Mor} L.J. Mordell: \emph{Diophantine Equations}, Academic Press 1969

\bibitem{Nag} K. Nagao: \emph{ An example of elliptic curve over $\Q(T)$ with rank $\geq 13$}, Proc. Japan Acad. Ser. A Math. Sci. 70 (1994), no. 5, 152--153

\bibitem{N} B. Nica: \emph{Jacobsthal Sums}, Monographs in Number Theory, World Scientific 2025

\bibitem{N2} B. Nica: \emph{Variance of point-counts for families of cubic curves over $\Fp$ and Jacobsthal sums}, preprint (2025), available at \url{https://arxiv.org/abs/2504.15505}


\bibitem{PAR} J.C. Parnami, M.K. Agrawal, A.R. Rajwade: \emph{Some identities involving character sums and their applications}, J. Indian Math. Soc. (N.S.) 54 (1989), no. 1-4, 125--132

\bibitem{RPR} D.B. Rishi, J.C. Parnami, A.R. Rajwade: \emph{Evaluation of a cubic character sum using the $\sqrt{-19}$ division points of the curve $Y^2=X^3-2^3\cdot 19 X+2\cdot 19^2$}, J. Number Theory 19 (1984), no. 2, 184--194

\bibitem{W79} K.S. Williams: \emph{Evaluation of character sums connected with elliptic curves}, Proc. Amer. Math. Soc. 73 (1979), no. 3, 291--299

\bibitem{Z} W. Zhang: \emph{Some interesting number theory problems}, preprint (2025), available at \url{https://arxiv.org/abs/2506.17235}

\end{thebibliography}
\end{document}